\documentclass[12pt,reqno]{amsart}
\usepackage{amsmath}
\usepackage{amssymb}
\usepackage{empheq}
\usepackage{comment}
\usepackage{amsthm}
\usepackage[top=1in, left=1in, right=0.9in, bottom=1in]{geometry}

\newtheorem{theorem}{Theorem}[section]
\newtheorem{lemma}[theorem]{Lemma}
\newtheorem{proposition}[theorem]{Proposition}
\newtheorem{corollary}[theorem]{Corollary}
\newtheorem{example}[theorem]{Example}
\newtheorem{conjecture}[theorem]{Conjecture}
\newtheorem{remark}[theorem]{Remark}

\begin{document}

\title{Equilibria of the field generated by point charges}

\author[Mykhailo Bilogliadov]{Mykhailo Bilogliadov}
\address{Department of Mathematics, Oklahoma State University, Stillwater, OK 74078, U.S.A.}
\email{mbilogli@math.okstate.edu}

\date{April 28, 2014}

\maketitle

\begin{abstract}
We consider a special case of Maxwell's problem on the number of equilibrium points of the Riesz potential $1/r^{2\beta}$ (where $r$ is the Euclidean distance and $\beta$ is the Riesz parameter) for positive unit point charges placed at the vertices of a regular polygon. We show that the equilibrium points are located on the perpendicular bisectors to the sides of the regular polygon, and study the asymptotic behavior of the equilibrium points with regard to the number of charges $n$ and the Riesz parameter $\beta$. Finally, we prove that for values of $\beta$ in a small neighborhood of $\beta=1$ the Riesz potential has only one equilibrium point different from the origin on each perpendicular bisector, and one equilibrium point at the origin.
\end{abstract}
\smallskip



\section{Introduction}

Consider a system of $n$ positive unit point charges in $\mathbb R^3$ located at points $x_i$, $i=1,2,\ldots,n$. They produce an electric field
\begin{equation*}
E(x)=\sum_{i=1}^n  \frac{x-x_i}{{| x-x_i|}^3}, \quad x_i,\, x\in\mathbb R^3.
\end{equation*}

Maxwell \cite{Maxwell} raised a question about the number of equilibrium points of the system of $n$ point charges, which is the number of points where $E(x)$ vanishes.
This is the same as the number of critical points of the Coulomb potential $U(x)$ associated with this electrostatic field $E(x)$. Maxwell conjectured
that the number of equilibrium points for the system of $n$ point charges has an upper bound $(n-1)^2$. In 2007, Gabrielov, Novikov, and Shapiro \cite{GNS} showed the
existence of an upper bound of $4^{n^2}(3n)^{2n}$ for any non-degenerate configuration of charges in $\mathbb R^3$. Killian \cite{Killian} improved the result obtained in \cite{GNS} in the case when all charges lie in a plane. Namely, he established that for $n$ charges in the plane
there is an upper bound of $(2^{n-1}(3n-2))^2$ on the number of equilibrium points. He also showed that in the case when unit point charges are located at the vertices of an equilateral triangle, there are exactly four equilibrium points. Peretz \cite{Peretz} proved that in the case of three positive point charges on a plane,  there must be at least two equilibrium points, and only an even number of them, i.e. 2, 4, 6, 8, 10 or 12. Also Tsai  \cite{Tsai} considered the case of three point charges of arbitrary signs placed at the vertices of isosceles and equilateral triangle, and obtained an upper bound of 4 on the number of equilibrium points.

In this paper we show that in the case when $n$ positive unit point charges are placed at the vertices of a regular $n$-gon, the Riesz potential $1/r^{2\beta}$ (where $r$ is the Euclidean distance) generated by this system has critical points located on the perpendicular bisectors to the sides of the polygon, as well as at the origin. For convenience, a critical point that is different from the origin will be called non-trivial. We also investigate the asymptotic behavior of the critical points with respect to $n$ and $\beta$. Namely, for a fixed $\beta$ and large $n$ we prove that all the non-trivial critical points become equidistributed on the unit circumference. When $n$ is fixed and $\beta$ is approaching zero, we prove that the non-trivial critical points slide to the origin, while for $\beta$ infinitely large we obtain that the distance from the origin to an equilibrium point on a bisector is bounded above by $\cos(\pi/n)$. Finally, using the techniques of the theory of stable mappings we show that for values of the parameter $\beta$ in a small neighborhood of $\beta=1$, the Riesz potential $1/r^{2\beta}$ has only one equilibrium point different from the origin on each perpendicular bisector, and one equilibrium point at the origin.

For simplicity we may further assume that charges are located in the $xy$-plane. The Riesz potential of this configuration is 
\begin{equation*}
U_\beta(x,y,z)=\sum_{j=1}^n \frac{1}{((x-x_j)^2+(y-y_j)^2+z^2)^{\beta}}.
\end{equation*}
We immediately see that the condition $\partial U_\beta(x,y,z)/\partial z=0$ yields $z=0$ for all $(x,y,z)\in\mathbb R^3$. Hence we only need to study
the critical points of the restriction of the 3D Riesz potential to the $xy$-plane. 


\section{Riesz potential of charges located at the roots of unity}
 
We note that the Riesz potential is a natural and straightforward generalization of the Coulomb potential. Therefore, we will be considering Maxwell's 
problem for the Riesz potential making  digressions to the Coulomb potential where it deems necessary.\\\indent
 Let $n$ positive unit point charges be placed at the vertices of a regular $n$-gon inscribed in a unit circumference. The Riesz potential of this system is
\begin{equation}\label{rieszpotential}
U_{\beta}(r,\theta)=\sum_{j=1}^n|z-\zeta_n^j|^{-2\beta},\quad n\geq3,\,\beta\in (0,1),
\end{equation}
where $\zeta_n=e^{2\pi i /n}$ is the $n$-th primitive root of unity and $z=re^{i\theta}$, with $r\in[0,1)$ and $\theta\in[0,2\pi)$.
We can rewrite the potential ($\ref{rieszpotential}$) as
\begin{equation}\label{rieszpotential1}
U_{\beta}(r,\theta)=\sum_{j=1}^n\frac{1}{(1+r^2-2r\cos(2\pi j/n-\theta))^{\beta}}.
\end{equation}

Note that for $\beta=1/2$, expression ($\ref{rieszpotential}$) reduces to the usual Coulomb potential
\begin{equation}\label{coulombpotential}
U(r,\theta)=\sum_{j=1}^n\frac{1}{(1+r^2-2r\cos(2\pi j/n-\theta))^{\frac{1}{2}}}.
\end{equation}

The following useful observation is due to Bang and Elmabsout  \cite{BangElmab}:

\begin{proposition}\label{bigbangtheorem}
Let $\alpha=q+\mu$, where $q$ is positive integer and $\mu\in(0,1)$. Then, for every $n\geq3$, real $\theta$ and real $r$ such that $0\leq r<1$  
\begin{empheq}{align}\label{stat1}
& \sum_{j=1}^n\frac{1}{(1+r^2-2r\cos(2\pi j/n-\theta))^{\alpha}}=\frac{n}{\Gamma(\alpha)\Gamma(1-\mu)}\times\\\nonumber
& \int_0^1\frac{t^{\alpha-1}}{(1-t)^{\mu}}\frac{\partial^q}{\partial t^q}\left\{\frac{t^q}{(1-tr^2)^{\alpha}}\frac{1-(tr)^{2n}}{1+(tr)^{2n}-2(tr)^n\cos(n\theta)}\right\}dt.
\end{empheq}
In the special case $\alpha=q+1$ we obtain that
\begin{empheq}{align}\label{stat2}
& \sum_{j=1}^n\frac{1}{(1+r^2-2r\cos(2\pi j/n-\theta))^{\alpha}}\\\nonumber
& =\left\{\frac{n}{q!}\frac{\partial^q}{\partial t^q}\left\{\frac{t^q}{(1-tr^2)^{\alpha}}\frac{1-(tr)^{2n}}{1+(tr)^{2n}-2(tr)^n\cos(n\theta)}\right\}\right\}_{t=1}.
\end{empheq}
\end{proposition}

As a corollary we obtain an important integral representation of the Riesz potential ($\ref{rieszpotential}$):
\begin{corollary}\label{prop1}
For $\beta\in(0,1)$, $r\in[0,1)$ and $0\leq\theta<2\pi$, the potential $U_{\beta}(r,\theta)$ can be written as
\begin{equation}\label{intrep}
U_{\beta}(r,\theta)= \frac{n\sin(\pi\beta)}{\pi}\int_0^1\frac{t^{\beta-1}(1-t)^{-\beta}}{(1-r^2t)^\beta}\frac{1-(rt)^{2n}}{1+(rt)^{2n}-2(rt)^n\cos(n\theta)}\,dt.
\end{equation}
\end{corollary}

Corollary $\ref{prop1}$ implies that in the case of Coulomb potential ($\beta=1/2$) we have
\begin{equation*}
U(r,\theta) = \frac{n}{\pi}\int_0^1\frac{t^{-1/2}(1-t)^{-1/2}}{(1-r^2t)^{1/2}}\frac{1-(rt)^{2n}}{1+(rt)^{2n}-2(rt)^n\cos(n\theta)}\,dt.
\end{equation*}


\section{Critical points of the Riesz potential}

Now we can use the integral representation of the Riesz potential $U_{\beta}(r,\theta)$ given in Corollary $\ref{prop1}$ to determine the location of the critical points of $U_{\beta}(r,\theta)$. It follows that the non-trivial critical points of the potential $U_{\beta}(r,\theta)$ are located on the perpendicular bisectors to the sides of a polygon, corresponding to the values of $\theta$ given by

\begin{equation}\label{critang}
\theta=\pi k/n,  \quad k=1,3,5\ldots,2n-1.
\end{equation}
More precisely we have
\begin{theorem}\label{locatcritpts}
Potential $U_{\beta}(r,\theta)$ has a non-trivial critical point on each perpendicular bisector to the sides of the regular polygon, and a critical point at the origin. Furthermore, each non-trivial critical point $(r,\theta)$ satisfies $$r\in(r_l(n,\beta), r_u(n)),$$ where
\begin{equation*}
r_l(n,\beta)=\left(\frac{\beta}{\beta+n}\right)^{\frac{1}{n-2}},\quad r_u(n)=\cos(\pi/n).
\end{equation*}
\end{theorem}

\begin{example}{($n=3$, $\beta=1/2$)}
Let $n=3$ and $\beta=1/2$. In this case, the charges are located at the vertices of an equilateral triangle inscribed in a unit circle centered at the origin. The Coulomb potential restricted to the bisectors  is $u_{1/2}(r)=2(r^2-r+1)^{-1/2}+(1+r)^{-1}$. It is easy to see that the critical points are in the interval $[0,1/2]$. Direct computation shows that the critical points are among the non-negative roots of the polynomial $r(r^5+5r^4+r^3+r^2-4r+1)$. Applying the Descarte's Rule of Signs to the polynomial $r^5+5r^4+r^3+r^2-4r+1$ we infer that there are two positive roots, one of which must be discarded as being outside of $[0,1/2]$. Thus there is a unique critical point for $u_{1/2}(r)$ different from the origin, and therefore the potential $U(r,\theta)$ has exactly four critical points. We remark that one of the two results in paper \cite{Tsai} dealt with the case of three point charges of arbitrary signs placed at the vertices of an equilateral triangle, but only an upper of 4 on the number of critical points was obtained.

\end{example}

\begin{example}{($n=4$, $\beta=1/2$)}
Let $n=4$ and $\beta=1/2$. Now the charges are placed at the vertices of a square inscribed in a unit circumference centered at the origin. For the Coulomb potential on the bisectors we obtain that $u_{1/2}(r)=2(r^2-\sqrt{2}r+1)^{-1/2}+2(r^2+\sqrt{2}r+1)^{-1/2}$. It follows that the critical points for $u_{1/2}(r)$ must be in the interval $[0,1/\sqrt{2}]$. It is not hard to see that the critical points are among the non-negative roots of the polynomial $r(4r^6+r^5-4r^3+1)$ and again applying the Descarte's Rule of Signs to the polynomial $4r^6+r^5-4r^3+1$ we conclude that it has two positive zeros one of which must be dropped as being outside of $[0,1/\sqrt{2}]$. Thus $u_{1/2}(r)$ has exactly two critical points, namely 0 and some $r_0\neq0$. This shows that $U(r,\theta)$ has exactly five critical points in the case of the square. 
\end{example}


\section{Asymptotic behavior of critical points}
It is not hard to see that the potential $U_\beta(r,\theta)$ has the same number of the critical points on each perpendicular bisector to the sides of a polygon. Indeed, as it follows
from ($\ref{intrep}$) and $(\ref{critang})$ the potential on the perpendicular bisectors is independent of $\theta$.
   
We also observe that the radial coordinate $r$ corresponding to an equilibrium position on each perpendicular bisector naturally depends on $n$ and $\beta$. We will show that when $\beta$ is fixed and the number  
of charges $n$ increases to infinity, $r(n,\beta)$ approaches $1$. That is, we have the following:
\begin{theorem}\label{unidistr}
Let $r(n,\beta)\in(0,1)$ be a radial coordinate of a non-trivial critical point on a perpendicular bisector to the sides. Then for a fixed $\beta\in(0,1)$ we have
\begin{equation}\label{limitvalueofr}
\lim_{n\rightarrow\infty}r(n,\beta)=1.
\end{equation}
\end{theorem}

Theorem $\ref{unidistr}$ shows that the critical points (except the one located at the origin) cluster on
the unit circumference for large $n$. We also see from ($\ref{critang}$) that for each fixed $n\geq 3$ 
all the non-trivial critical points consist of $n$-tuple subsets equidistributed on circumference with radius
$r(n,\beta)$. We deduce that 
when $n$ becomes infinitely large all the non-trivial critical points become uniformly distributed on the unit circumference.

Now we investigate the dependence of $r(n,\beta)$ on the parameter $\beta$ when $n$ is fixed. We will consider the asymptotic behavior of critical points when $\beta\rightarrow0+$ and $\beta\rightarrow+\infty$. First we treat the case $\beta\rightarrow0+$.
\begin{theorem}\label{betasmall}
Let $n\geq3$ be fixed. Let $r(n,\beta)$ be a radius corresponding to an equilibrium position different from the origin. Then
\begin{equation}\label{deponbeta0}
\lim_{\beta\rightarrow0+}r(n,\beta)=0.
\end{equation}
\end{theorem}

We turn to the case $\beta\rightarrow+\infty$ now. We want to show that when $\beta$ increases, the distance from the origin to the critical points is restricted by a side of a regular polygon. In other words, an equilibrium position $r(n,\beta)$ can not exceed (the length of) an apothem of a regular polygon. In our case we have a regular polygon inscribed in a unit circumference, therefore its apothem is just $\cos(\pi/n)$.
\begin{theorem}\label{betalarge}
Assume $n\geq3$ is fixed and let $r(n,\beta)$ be a radius corresponding to an equilibrium position different from the origin. Under these assumptions we have
\begin{equation}\label{bigbetatheorem}
\limsup_{\beta\rightarrow+\infty}r(n,\beta)\leq\cos\left(\frac{\pi}{n}\right).
\end{equation}
\end{theorem}

In fact, after considering numerous examples, one arrives at the following:

\begin{conjecture} Under the assumptions of Theorem $\ref{betalarge}$
\begin{equation*}
\lim_{\beta\rightarrow+\infty}r(n,\beta)=\cos\left(\frac{\pi}{n}\right).
\end{equation*}
\end{conjecture}


\section{Number of critical points of the Riesz potential $U_\beta(r,\theta)$ in the neighborhood of $\beta=1$}

In the case $\beta=1$ the Riesz potential ($\ref{rieszpotential1}$) takes the form
\begin{equation}\label{limcase}
U_1(r,\theta)=\sum_{j=1}^n\frac{1}{1+r^2-2r\cos(2\pi j/n-\theta)}.
\end{equation}
What is rather surprising is that the potential $U_1(r,\theta)$ admits the following simple closed expression:
\begin{proposition}\label{simplerep}
For $r\in[0,1)$ and $0\leq\theta<2\pi$
\begin{equation}\label{finsum}
U_1(r,\theta)=\frac{n}{1-r^2}\frac{1-r^{2n}}{1+r^{2n}-2r^n\cos(n\theta)}.
\end{equation}
\end{proposition}

Now it is not hard to count the number of critical points of the potential $U_1(r,\theta)$.

\begin{theorem}\label{critptsbeta1}
The potential $U_1(r,\theta)$ has exactly $n+1$ critical points. Namely, it has exactly one critical point on each perpendicular bisector to the sides of the polygon, and a critical point at the origin. 
\end{theorem}

Next, we observe that the restriction of the potential $U_\beta(r,\theta)$ to a perpendicular bisector to the sides is an analytic function in $\beta,r$. This fact, coupled with the theory of stable mappings shows that the property of the potential $U_\beta(r,\theta)$ 
to have exactly $n+1$ critical points is preserved when $\beta$ varies in a small neighborhood of $\beta=1$.
\begin{theorem}\label{extres}
There exist real numbers $\beta_l$ and $\beta_r$ such that $1\in(\beta_l,\beta_r)$ so that for each $\beta\in(\beta_l,\beta_r)$ the potential $U_\beta(r,\theta)$ has $n+1$ critical points. In particular, for each $\beta\in(\beta_l,\beta_r)$ the potential $U_\beta(r,\theta)$ has only one critical point different from the origin on each perpendicular bisector to the sides of the polygon, and a critical point at the origin.
\end{theorem}

In fact, as the next example clearly demonstrates, Maxwell's conjecture for the Riesz potentials holds for three unit positive point charges placed at the vertices of an equilateral triangle {\it for all} $\beta\in (0,1)$.
\begin{example}{($n=3$, $\beta\in(0,1)$)}\label{ex3}
Let $n=3$ and $\beta\in(0,1)$. Then the Riesz potential $U_\beta(r,\theta)$ has exactly four critical points. Namely, for all $\beta\in(0,1)$ there is one critical point on each perpendicular bisector to the sides of a triangle, and a critical point at the intersection of the perpendicular bisectors. The proof of this statement is provided at the Proofs section.
\end{example}

\section{Proofs}


\subsection{Proofs for Section 2}

{\bf Proof of Corollary $\ref{prop1}$.}
Proof of Corollary $\ref{prop1}$ can be found in \cite{FerrPorta} or \cite{BangElmab}, where it arises in connection with a polygonal many-body problem.
There is similarity between the ideas arising in the theory of many-body problem and electrostatics, in particular due to the fact that the potentials
which appear in both are identical. The idea behind Corollary $\ref{prop1}$ has its origin in the book of Tisserand \cite{Tisserand}
on celestial mechanics, and later it appears more or less explicitly in the paper of Lindow \cite{Lindow}. For the sake of self-containment we will reproduce the proof from the paper of Ferrario and Portaluri \cite{FerrPorta}.

First we observe that potential ($\ref{rieszpotential}$) can be rewritten in the following manner:
\begin{equation}\label{rieszpotential2}
U_{\beta}(r,\theta)=\sum_{j=1}^n|1-z\zeta_n^j|^{-2\beta},\quad n\geq3,\,\beta\in (0,1).
\end{equation}
This follows from a fact  that if $\zeta_n$ is the $n$-th primitive root of unity, we have that $\zeta_n^{-1}=1/\zeta_n$ is an $n$-th root of unity, and an observation $|1-z\zeta_n^j|=|z-\zeta_n^{-j}|$. In order to transform the right-hand side of ($\ref{rieszpotential2}$), we expand $|1-r\xi|^{-2\beta}$ in a double power series as follows:
\begin{lemma}
For each $r\in[0,1)$, $\xi=e^{i\theta}$, $0\leq\theta<2\pi$, and $\beta>0$, we have the expansion
\begin{equation*}
|1-r\xi|^{-2\beta}=\sum_{m=-\infty}^{+\infty}b_m\xi^m,
\end{equation*}
\noindent
where
\begin{equation*}
b_m=b_{-m}=\frac{\sin(\beta\pi)}{\pi}r^m\int_0^1(1-t)^{-\beta}t^{\beta-1}t^m(1-tr^2)^{-\beta}\,dt,\, m\geq 0.
\end{equation*}
\end{lemma}
\begin{proof}
We have
\begin{empheq}{align*}
|1-r\xi|^{-2\beta}&=\ \ (1-r\xi)^{-\beta}(1-r\xi^{-1})^{-\beta}\\
&=\left(\sum_{k=0}^{\infty}\left(\begin{array}{c} -\beta \\ k \end{array} \right) (-r\xi)^k\right)\cdot
\left(\sum_{h=0}^{\infty}\left(\begin{array}{c} -\beta \\ h \end{array} \right) (-r\xi^{-1})^h\right)\\
&=\sum_{h,k=0}^{\infty}\left(\begin{array}{c} -\beta \\ k \end{array} \right)\left(\begin{array}{c} -\beta \\ h \end{array} \right) (-r)^{k+h}\xi^{k-h}\\
&=\sum_{m=-\infty}^{\infty}\left((-1)^m\sum_{k-h=m;\ \ k,h\geq0}\left(\begin{array}{c} -\beta \\ k \end{array} \right)\left(\begin{array}{c} -\beta \\ h \end{array} \right) r^{k+h}\right)
\xi^m.
\end{empheq}
\noindent
Let
\begin{equation*}
b_m=(-1)^m\sum_{k-h=m;\ \ k,h\geq0}\left(\begin{array}{c} -\beta \\ k \end{array} \right)\left(\begin{array}{c} -\beta \\ h \end{array} \right) r^{k+h}.
\end{equation*}
\noindent
Next, recall that for each $\beta>0$ and $N\in\mathbb N$ we have
\begin{empheq}{align*}
\left(\begin{array}{c} -\beta \\ N \end{array} \right)&=(-1)^N\left(\begin{array}{c} N+\beta-1 \\ N \end{array} \right)\\
&=(-1)^N\frac{\Gamma(N+\beta)\Gamma(1-\beta)}{\Gamma(N+1)\Gamma(\beta)\Gamma(1-\beta)}\\
&=\frac{(-1)^N}{\Gamma(\beta)\Gamma(1-\beta)}\frac{\Gamma(N+\beta)\Gamma(1-\beta)}{\Gamma(N+1)}\\
&=\frac{(-1)^N\sin(\pi\beta)}{\pi}B(1-\beta,N+\beta),
\end{empheq}
\noindent
where $B(x,y)$ is the Beta function and $\Gamma(x,y)$ is the Gamma function. Here we have used the following identities \cite{AAR}:
\begin{empheq}{align*}
&\Gamma(\beta)\Gamma(1-\beta)=\frac{\pi}{\sin(\beta\pi)},\\
&\left(\begin{array}{c} -\beta \\ N \end{array} \right)=(-1)^N\frac{\Gamma(N+\beta)}{\Gamma(N+1)\Gamma(\beta)},\\
&\left(\begin{array}{c} \beta \\ N \end{array} \right)=\frac{\Gamma(1+\beta)}{\Gamma(N+1)\Gamma(\beta-N+1)}.
\end{empheq}
\noindent
Using the integral representation for the Beta function \cite[p. 4]{AAR} we obtain
\begin{equation*}
\left(\begin{array}{c} -\beta \\ N \end{array} \right)=(-1)^N\frac{\sin(\beta\pi)}{\pi}\int_0^1(1-t)^{-\beta}t^{\beta-1}t^Ndt.
\end{equation*}
\noindent
This implies that for $N=h+m$
\begin{empheq}{align*}
b_m& = (-1)^m\sum_{h=0}^{\infty}\left((-1)^{m+h}\frac{\sin(\beta\pi)}{\pi}\int_0^1(1-t)^{-\beta}t^{\beta-1}t^{m+h}dt
\left(\begin{array}{c} -\beta \\ h \end{array} \right) r^{m+2h}\right)\\
& =\frac{\sin(\beta\pi)}{\pi}r^m\sum_{h=0}^{\infty}\left((-1)^h\int_0^1(1-t)^{-\beta}t^{\beta-1}t^{m+h}dt \left(\begin{array}{c} -\beta \\ h \end{array} \right)
r^{2h}\right)\\
& =\frac{\sin(\beta\pi)}{\pi}r^m\int_0^1(1-t)^{-\beta}t^{\beta-1}t^m\left(\sum_{h=0}^{\infty}(-1)^h t^h \left(\begin{array}{c} -\beta \\ h \end{array} \right)
r^{2h}\right)dt\\
&= \frac{\sin(\beta\pi)}{\pi}r^m\int_0^1(1-t)^{-\beta}t^{\beta-1}t^m(1-tr^2)^{-\beta}dt.
\end{empheq}
\noindent
Here, we used the fact that
\begin{equation*}
\sum_{h=0}^{\infty}(-tr^2)^h \left(\begin{array}{c} -\beta \\ h \end{array} \right)=(1-tr^2)^{-\beta}.
\end{equation*}
\end{proof}
\noindent
Now observe that for each $k\in\mathbb Z$
\begin{equation*}
\frac{1}{n}\sum_{y:y^n=\xi}y^k=
\left\{\begin{array}{cl}
	\xi^{k/n}, & \text{if}\ \ k\equiv0\ \ \text{mod}\ \ n, \\
	0, & \text{if}\ \ k\not\equiv0\ \ \text{mod}\ \ n,
	   \end{array}\right.
\end{equation*}
\noindent
so for every $\beta\in(0,1)$, $r\in\ [0,1)$ and integer $n\geq2$, we obtain that
\begin{equation*}
U_{\beta}(r,\theta)=n\sum_{k=-\infty}^{\infty}b_{nk}\xi^k.
\end{equation*}
\noindent
Next, we have
\begin{empheq}{align*}
\sum_{k=-\infty}^{\infty}b_{nk}\xi^k & = \sum_{k=0}^{\infty}b_{nk}\xi^k + \sum_{k=1}^{\infty}b_{nk}\xi^{-k}\\
& = \sum_{k=0}^{\infty}\left(\frac{\sin(\beta\pi)}{\pi}r^{nk}\int_0^1(1-t)^{-\beta}t^{\beta-1}t^{nk}(1-tr^2)^{-\beta}dt\right)\xi^k\\
& + \sum_{k=1}^{\infty}\left(\frac{\sin(\beta\pi)}{\pi}r^{nk}\int_0^1(1-t)^{-\beta}t^{\beta-1}t^{nk}(1-tr^2)^{-\beta}dt\right)\xi^{-k}\\
& = \frac{\sin(\beta\pi)} {\pi} \int_0^1\frac{(1-t)^{-\beta}t^{\beta-1}} {(1-tr^2)^\beta}\left[\sum_{k=0}^{\infty}(tr)^{nk}\xi^k+\sum_{k=1}^{\infty}(tr)^{nk}\xi^{-k}\right]dt.
\end{empheq}
\noindent
Since
\begin{equation*}
\sum_{k=0}^{\infty}(tr)^{nk}\xi^k+\sum_{k=1}^{\infty}(tr)^{nk}\xi^{-k}=\frac{1-(tr)^{2n}}{|1-(tr)^n\xi|^2},
\end{equation*}
\noindent
we arrive at ($\ref{intrep}$):
\begin{equation*}
U_{\beta}(r,\theta)=\frac{n\sin(\beta\pi)} {\pi}\int_0^1\frac{(1-t)^{-\beta}t^{\beta-1}}{(1-tr^2)^{\beta}}\frac{1-(tr)^{2n}}{|1-(tr)^n\xi|^2}\,dt.
\end{equation*}
\qed


\subsection{Proofs for Section 3}

{\bf Proof of Theorem $\ref{locatcritpts}$.} First we find the critical values of $\theta$ by using the integral representation $(\ref{intrep})$ given in Corollary $\ref{prop1}$. The derivative with respect to $\theta$ is
\begin{empheq}{align*}
\frac{\partial U_{\beta}(r,\theta)}{\partial \theta}=&-\sin(n\theta)\frac{2n^2r^n\sin(\pi\beta)}{\pi}\times\\
&\int_0^1\frac{t^{n+\beta-1}(1-t)^{-\beta}}{(1-r^2t)^\beta}\frac{1-(rt)^{2n}}{(1+(rt)^{2n}-2(rt)^n\cos(n\theta))^2}\,dt.
\end{empheq}\noindent
Then if $\theta$ is a critical value of the potential, we have $\partial U_{\beta}(r,\theta)/\partial \theta=0$. This implies that $\sin(n\theta)=0$, so that $n\theta=\pi k, k\in\mathbb Z$,
or
 \begin{equation}\label{critval}
\theta=\pi k/n, \quad k\in\mathbb Z,
\end{equation}
because the above integral is positive.

Consider the case $k=0$ in ($\ref{critval}$). We see that there are possible critical points on a ray coming from the origin where our regular $n$-gon is
centered, and going through a vertex which sits on the $x$-axis. Now letting $\theta=0$ in ($\ref{intrep}$), we see the potential on that ray is given by
\begin{equation}\label{Uonbadsectors}
U_{\beta}(r,0)=\frac{n\sin(\pi\beta)}{\pi}\int_0^1\frac{t^{\beta-1}(1-t)^{-\beta}}{(1-r^2t)^\beta}\frac{1+(rt)^n}{1-(rt)^n}\,dt.
\end{equation}
\noindent
The critical points of $U_{\beta}(r,0)$ are found from
\begin{equation}
\frac{\partial U_{\beta}(r,0)}{\partial r}=0.
\end{equation}
Since the integrand on the right-hand side of $(\ref{Uonbadsectors})$ and its derivative are continuous with respect to $(r,t)\in[0,x_0)\times[0,1]$, for any $x_0\in[0,1)$, we can differentiate under the sign of the integral \cite[p. 236]{Rudin}. Then for the derivative of $U_{\beta}(r,0)$ we obtain
\begin{empheq}{align}\label{deri}
&\frac{\partial U_{\beta}(r,0)}{\partial r}=\frac{n\sin(\pi\beta)}{\pi}\times\\ \nonumber
&\int_0^1 t^{\beta-1}(1-t)^{-\beta}\left\{\frac{2\beta rt}{(1-r^2t)^{\beta+1}}\frac{1+(rt)^n}{1-(rt)^n}+\frac{1}{(1-r^2t)^\beta}\frac{2nr^{n-1}t^n}{(1-(rt)^n)^2}\right\}\,dt.
\end{empheq}
\noindent
For $r\in(0,1)$ and $t\in[0,1]$, we see that
\begin{equation*}
\frac{1}{1-t}\geq1,\ \ \frac{1}{1-tr^2}\geq1,\ \  \frac{1}{1-r^nt^n}\geq1, \ \  \frac{1+r^nt^n}{1-r^nt^n}\geq1.
\end{equation*}
\noindent
Therefore
\begin{empheq}{align*}
\frac{\partial U_{\beta}(r,0)}{\partial r}&\geq \frac{n\sin(\pi\beta)}{\pi}\int_0^1 t^{\beta-1} \left( 2\beta rt+2nr^{n-1}t^n\right)dt\\
&= \frac{n\sin(\pi\beta)}{\pi}\left\{\frac{2\beta r}{\beta+1}+\frac{2nr^{n-1}}{n+\beta}\right\}>0.
\end{empheq}
\noindent
Hence $\partial U_{\beta}(r,0)/\partial r$ is a strictly positive function for $r\in(0,1)$. Now looking at the expression ($\ref{deri}$) for $\partial U_{\beta}(r,0)/\partial r$
it is clear that it can be zero if and only if $r=0$. Thus for $r\in[0,1)$ the potential on that ray has only one critical point, namely at the origin. This also shows that the potential $u_\beta(r)$ does not have the critical points different from the origin for the values of $\theta=\pi k/n$, corresponding to the even values of $k\in\mathbb Z$.

So far we have showed that all the critical points of the Riesz potential ($\ref{rieszpotential}$) (if they exist) lie on the rays $\theta=\pi k/n$, where $k=1,3,5,\ldots, 2n-1$.
We now want to obtain further information on the critical points on those rays. For that we let $\theta=\pi k/n,  k=1,3,5,\ldots, 2n-1$ in ($\ref{rieszpotential1}$) and consider the case
$k=1$, which is enough for our purposes. We therefore obtain
\begin{equation}\label{potonbis}
u_\beta(r):=U_{\beta}(r,\pi/n)=\frac{n\sin(\pi\beta)}{\pi}J_n^{\beta}(r),
\end{equation}
\noindent
where
\begin{equation}\label{functionJ}
J_n^{\beta}(r):=\int_0^1\frac{t^{\beta-1}(1-t)^{-\beta}}{(1-r^2t)^\beta}\frac{1-(rt)^{n}}{1+(rt)^n}\,dt.
\end{equation}
\noindent
We are interested in the critical points of $J_n^{\beta}(r)$:
\begin{equation*}
\frac{dJ_n^{\beta}(r)}{dr}=0,\ \ n\geq3.
\end{equation*}
\noindent
Clearly the integrand in ($\ref{functionJ}$) and its derivative are continuous with respect to $(r,t)\in[0,x_0)\times[0,1]$, for any $x_0\in[0,1)$. Hence we can differentiate under the sign of the integral in ($\ref{functionJ}$)  \cite[p. 236]{Rudin}.\\\indent
Now let us show that $u_\beta'(r)>0$ for $r$ from a small right-hand side neighborhood of $0$, for all $n\geq3$. It is clear that
\begin{equation}\label{derexprssn}
\frac{dJ_n^{\beta}(r)}{dr}=2r\int_0^1\frac{t^{\beta}(1-t)^{-\beta}}{(1-r^2t)^{\beta+1}}\frac{g_n(r,t)}{(1+(rt)^n)^2}\,dt,\quad n\geq3,
\end{equation}
where
\begin{equation}\label{functiong}
g_n(r,t)=\beta(1-(rt)^{2n})-nr^{n-2}t^{n-1}(1-r^2t).
\end{equation}

We can estimate $g_n(r,t)$ from below in the following manner:

\begin{empheq}{align*}
g_n(r,t)&=\beta(1-(rt)^{2n})-nr^{n-2}t^{n-1}(1-r^2t)\\
&=\beta-\beta r^{2n}t^{2n}-nr^{n-2}t^{n-1}+nr^nt^n\\
&\geq\beta-\beta r^{2n}-nr^{n-2}\\
&\geq\beta-\beta r^{n-2}-nr^{n-2}\\
&=\beta- r^{n-2}(\beta+n),
\end{empheq}
for all $t\in[0,1]$, all $n\geq3$ and all $\beta\in(0,1)$. Thus if we take
\begin{equation}\label{lowerb}
r<r_l(n,\beta):=\left(\frac{\beta}{\beta+n}\right)^{\frac{1}{n-2}}
\end{equation}
it follows that  $g_n(r,t)>0$ for all $t\in[0,1]$ and all $n\geq3$.\\

Next, we note that the function $t^{\beta}(1-t)^{-\beta}(1-r^2t)^{-(\beta+1)}$ is continuous and nonnegative for all $0\leq t\leq 1$. Therefore
by the second Mean Value Theorem for integrals there exists $t^*\in[0,1]$ such that
\begin{empheq}{align*}
&\int_0^1\frac{t^{\beta}(1-t)^{-\beta}}{(1-r^2t)^{\beta+1}}\frac{g_n(r,t)}{(1+(rt)^n)^2}\,dt\\
&=\frac{g_n(r,t^*)}{(1+(rt^*)^{n})^2}\int_0^1\frac{t^{\beta}(1-t)^{-\beta}}{(1-r^2t)^{\beta+1}}\,dt\\
&=\frac{g_n(r,t^*)}{(1+(rt^*)^n)^2} B(1+\beta,1-\beta) {}_2 F_1(1+\beta,1+\beta;2;r^2),
\end{empheq}
where ${}_2 F_1$ is the Gauss hypergeometric function \cite{AAR}.
\noindent
Therefore we see that if $0<r<r_l(n,\beta)$ then $u_\beta'(r)>0$ for all $n\geq3$. One can also check that
$u_\beta'(r_l(n,\beta))\neq0$. Thus proved the following:
\begin{lemma}\label{leftinterv}
Potential $u_\beta(r)$ has no critical points on $(0,r_l(n,\beta))$.
\end{lemma}

To obtain an upper estimate on the location of a critical point on a bisector we recall that the Riesz potential on a bisector for $0\leq r<1$ and $n\geq3$ is given by
\begin{equation}\label{trig1}
u_{\beta}(r)=\sum_{j=1}^n\frac{1}{(1+r^2-2r\cos(\theta_j))^{\beta}},
\end{equation}
\noindent
where $\theta_j=\pi(2j-1)/n$, $j=1,2,\ldots,n$.\\\indent
For the derivative $u_\beta'(r)$ we find
\begin{equation}\label{trig2}
u_{\beta}'(r)=-2\beta\sum_{j=1}^n\frac{r-\cos(\theta_j)}{(1+r^2-2r\cos(\theta_j))^{\beta+1}}.
\end{equation}
so that 
\begin{equation*}
u_{\beta}'(\cos(\pi/n))=-2\beta\sum_{j=1}^n\frac{\cos(\pi/n)-\cos(\theta_j)}{(1+\cos^2(\pi/n)-2\cos(\pi/n)\cos(\theta_j))^{\beta+1}}.
\end{equation*}
\noindent
It is now obvious that $u_{\beta}'(\cos(\pi/n))<0$ for all $n\geq3$. Let $r_u(n):=\cos(\pi/n)$.
\\\indent
Since $u_{\beta}'(r)$ is continuous and changes sign on $(0,r_u(n))$, we infer that there exists $r_0\in(0,r_u(n))$ such that $u_{\beta}'(r_0)=0$. Also from $(\ref{trig2})$ it is clear that $u_{\beta}'(r)<0$ for $r\in[r_u(n),1)$. Hence we proved the following:
\begin{lemma}\label{rightinterv}
Potential $u_\beta(r)$ does not have critical points on $[r_u(n),1)$.
\end{lemma}
Combining Lemma $\ref{leftinterv}$ and Lemma $\ref{rightinterv}$ we obtain
\begin{lemma}\label{intervcritpoints}
Potential $u_{\beta}(r)$ has all its critical points in the interval $(r_l(n,\beta)$, $r_u(n))$.
\end{lemma}
Since the potential $u_\beta(r)$ is a restriction of 2D potential  $U_{\beta}(r,\theta)$ to the perpendicular bisectors, the proof of the theorem follows.
\qed

\begin{remark}
We note that for the potential $u_\beta(r)$ a brief calculation reveals that $0$ is a critical point and
$u''_\beta(0)=2\beta^2n$. Therefore $0$ is a local minimum for the potential $u_\beta(r)$.
\end{remark}


\subsection{{\bf Proofs for Section 4}}

{\bf Proof of Theorem $\ref{unidistr}$.}
We know that $r(n,\beta)\in(r_l(n,\beta),r_u(n))$ for all $\beta\in(0,1)$ and all $n\geq3$. Note that
\begin{equation*}
\lim_{n\rightarrow\infty}r_l(n,\beta)=\lim_{n\rightarrow\infty}\left(\frac{\beta}{\beta+n}\right)^{\frac{1}{n-2}}=1,
\end{equation*}
\begin{equation*}
\lim_{n\rightarrow\infty}r_u(n)=\lim_{n\rightarrow\infty}\cos\left(\frac{\pi}{n}\right)=1.
\end{equation*}
Therefore
\begin{equation*}
\lim_{n\rightarrow\infty}r(n,\beta)=1.
\end{equation*}
\qed

{\bf Proof of Theorem $\ref{betasmall}$.}
Without loss of generality we may assume that $\beta\in(0,1/2)$. Recall that according to $(\ref{potonbis})$ the potential $u_\beta(r)$ has the form
\begin{equation*}
u_\beta(r)=\frac{n\sin(\pi\beta)}{\pi}\int_0^1\frac{t^{\beta-1}(1-t)^{-\beta}}{(1-r^2t)^\beta}\frac{1-(rt)^{n}}{1+(rt)^n}dt, 
\end{equation*}
and so
\begin{empheq}{align}\label{est_1}
u'_\beta(r)=&\frac{2nr\sin(\pi\beta)}{\pi}\left\{\beta\int_0^1\frac{t^{\beta}(1-t)^{-\beta}}{(1-r^2t)^{\beta+1}}\frac{1-(rt)^n}{1+(rt)^n}\,dt\right.\\\nonumber
-&\left.nr^{n-2}\int_0^1\frac{t^{n+\beta-1}(1-t)^{-\beta}}{(1-r^2t)^{\beta}}\frac{1}{(1+(rt)^n)^2}\,dt\right\}.
\end{empheq}
We want to estimate the integrals on the right-hand side of $(\ref{est_1})$. For that a few simple inequalities will prove handy.\\\indent 
Observe that $1-(rt)^n=(1-(rt))(1+rt+(rt)^2+\ldots+(rt)^{n-1})$. Hence, as $r\in(0,1)$ and $t\in[0,1]$ we deduce that $1-(rt)^n=(1-(rt))(1+rt+(rt)^2+\ldots+(rt)^{n-1})\leq(1-r^2t)(1+1+\ldots+1)=n(1-r^2t)$. In addition, $(1-r^2t)^{-\beta}\leq(1-t)^{-\beta}$ and, trivially $(1+r^nt^n)^{-1}\leq1$. Therefore
\begin{empheq}{align}\label{est_2}
\int_0^1\frac{t^{\beta}(1-t)^{-\beta}}{(1-r^2t)^{\beta+1}}\frac{1-(rt)^n}{1+(rt)^n}\,dt&=\int_0^1\frac{t^{\beta}(1-t)^{-\beta}}{(1-r^2t)(1-r^2t)^{\beta}}\frac{1-(rt)^n}{1+(rt)^n}\,dt\\\nonumber
&\leq\int_0^1\frac{t^{\beta}(1-t)^{-\beta}}{(1-r^2t)(1-t)^{\beta}}\frac{n(1-r^2t)}{1+(rt)^n}\,dt\\\nonumber
&\leq n\int_0^1t^\beta(1-t)^{-2\beta}\,dt\\\nonumber
&=nB(1+\beta,1-2\beta),&
\end{empheq}
where $B(x,y)$ is the beta-function.

Also, the second integral on the right-hand side of $(\ref{est_1})$ can be estimated as follows:

\begin{empheq}{align}\label{est_3}
\int_0^1\frac{t^{n+\beta-1}(1-t)^{-\beta}}{(1-r^2t)^{\beta}}\frac{1}{(1+(rt)^n)^2}\,dt &\geq\frac{1}{4}\int_0^1t^{n+\beta-1}\,dt \\\nonumber
&=\frac{1}{4(n+\beta)}.&
\end{empheq}

Inequalities $(\ref{est_2})$ and $(\ref{est_3})$ imply that
\begin{equation}\label{est_4}
u'_\beta(r)\leq\frac{2n^2r\sin(\pi\beta)}{\pi}\left\{\beta B(1+\beta,1-2\beta)-\frac{r^{n-2}}{4(n+\beta)}\right\}.
\end{equation}
Let
\begin{equation}\label{rupperbd}
r_u(n,\beta):=\left\{4\beta (n+\beta)B(1+\beta,1-2\beta)\right\}^{\frac{1}{n-2}}.
\end{equation}
Then, as it follows from $(\ref{est_4})$, if $r>r_u(n,\beta)$ we have $u'_\beta(r)<0$. This shows that there are no critical points on the right from $r_u(n,\beta)$. Hence all the critical points belong to the interval $(r_l(n,\beta),r_u(n,\beta))$.\\\indent
Note that ($n\geq3$ is fixed)
\begin{equation*}
\lim_{\beta\rightarrow0+}r_u(n,\beta)=\lim_{\beta\rightarrow0+}\left\{4\beta (n+\beta)B(1+\beta,1-2\beta)\right\}^{\frac{1}{n-2}}=0,
\end{equation*}
and also
\begin{equation*}
\lim_{\beta\rightarrow0+}r_l(n,\beta)=\lim_{\beta\rightarrow0+}\left(\frac{\beta}{\beta+n}\right)^{\frac{1}{n-2}}=0.
\end{equation*}
Therefore, since $r_l(n,\beta)<r(n,\beta)<r_u(n,\beta)$ clearly
\begin{equation*}
\lim_{\beta\rightarrow0+}r(n,\beta)=0,
\end{equation*}
as desired.
\qed

{\bf Proof of Theorem $\ref{betalarge}$.}
We will need a finite sum representation for the Riesz potential on a perpendicular bisector. It is given by
\begin{equation*}
u_{\beta}(r)=\sum_{j=1}^n\frac{1}{(1+r^2-2r\cos(\pi(2j-1)/n))^{\beta}}.
\end{equation*}
Differentiating, we obtain
\begin{equation*}
u_{\beta}'(r)=-2\beta\sum_{j=1}^n\frac{r-\cos(\pi(2j-1)/n)}{(1+r^2-2r\cos(\pi(2j-1)/n))^{\beta+1}}.
\end{equation*}
From the above formula it is clear that if $r(n,\beta)$ is an equilibrium point, then for any $n\geq3$ and $\beta>0$ we have un upper estimate
\begin{equation}\label{uniest}
r(n,\beta)<\cos\left(\frac{\pi}{n}\right).
\end{equation}
The desired result then follows by passing to the $\limsup_{\beta\rightarrow+\infty}$ in the both sides of the inequality in $(\ref{uniest})$.
\qed


\subsection{{\bf Proofs for Section 5}}

{\bf Proof of Proposition $\ref{simplerep}$.} We will use a simple argument which can be found, for example, in \cite{ChuMar}. Consider the following cyclotomic
polynomial
\begin{equation*}
1-z^n=\prod_{j=1}^{n}(1-z\zeta_n^j),
\end{equation*}
where $\zeta_n=e^{2\pi i/n}$ is the $n$-th primitive root of unity. Consider a partial fraction decomposition for $(1-z^n)^{-1}$:
\begin{equation*}
\frac{1}{1-z^n}=\frac{1}{n}\sum_{j=1}^n\frac{1}{1-z\zeta_n^j}.
\end{equation*}
Hence
\begin{equation}\label{cyclo1}
\sum_{j=1}^n\frac{1}{1-z\zeta_n^j}=\frac{n}{1-z^n}.
\end{equation}
Now replace $z$ by $1/z$ in $(\ref{cyclo1})$. We find that
\begin{equation*}
\frac{z^nn}{z^n-1}=\sum_{j=1}^n\frac{z}{z-\zeta_n^j}=-\sum_{j=1}^n\frac{z\overline{\zeta_n^j}}{1-z\overline{\zeta_n^j}}=-\sum_{j=1}^n\frac{z\zeta_n^j}{1-z\zeta_n^j},
\end{equation*}
where we used the fact that $\overline{\zeta_n}$ is an $n$-th root of unity.  Thus
\begin{equation}\label{cyclo2}
\sum_{j=1}^n\frac{z\zeta_n^j}{1-z\zeta_n^j}=\frac{z^nn}{1-z^n}.
\end{equation}
Let $z=re^{i\theta}$. Taking the real part in $(\ref{cyclo2})$ we obtain
\begin{equation}\label{cyclo3}
\sum_{j=1}^{n}\frac{r^2-r\cos(2\pi j/n+\theta)}{1+r^2-2r\cos(2\pi j/n+\theta)}=\frac{nr^{2n}-nr^n\cos(n\theta)}{1+r^{2n}-2r^n\cos(n\theta)},
\end{equation}
while taking the real part in $(\ref{cyclo1})$ yields
\begin{equation}\label{cyclo4}
\sum_{j=1}^n\frac{1-r\cos(2\pi j/n+\theta)}{1+r^2-2r\cos(2\pi j/n+\theta)}=\frac{n-nr^n\cos(n\theta))}{1+r^{2n}-2r^n\cos(n\theta)}.
\end{equation}
Subtracting $(\ref{cyclo3})$ from $(\ref{cyclo4})$ we obtain that
\begin{equation*}
\sum_{j=1}^n\frac{1-r^2}{1+r^2-2r\cos(2\pi j/n+\theta)}=\frac{n(1-r^{2n})}{1+r^{2n}-2r^n\cos(n\theta)}.
\end{equation*}
We finally obtain that
\begin{equation*}
U_1(r,\theta)=\sum_{j=1}^n\frac{1}{1+r^2-2r\cos(2\pi j/n-\theta)}=\frac{n}{1-r^2}\frac{1-r^{2n}}{1+r^{2n}-2r^n\cos(n\theta)},
\end{equation*}
as claimed.
\qed

{\bf Proof of Theorem $\ref{critptsbeta1}$.} It is easy to see that the critical values of the angle $\theta$ are given by $\theta=\pi k/n$, where 
$k\in\mathbb Z$. Next, we sort out even and odd values of $k$. First, we have the following:
\begin{lemma}\label{crtianglesbeta1}
For $r\in[0,1)$, the potential $U_1(r,\theta)$ on a ray corresponding to even  $k$ has only one critical point, namely the origin. 
\end{lemma}
\begin{proof}
It is sufficient to consider the case $k=0$. We readily find
\begin{equation*}
U_1(r,0)=\frac{n}{1-r^2}\,\frac{1+r^n}{1-r^n}.
\end{equation*}
Then, differentiating we obtain
\begin{equation*}
\frac{\partial U_1(r,0)}{\partial r}=n\left\{\frac{2r}{(1-r^2)^2}\,\frac{1+r^n}{1-r^n}+\frac{2n}{1-r^2}\,\frac{r^{n-1}}{(1-r^n)^2}\right\}.
\end{equation*}
Now trivial estimates show that $\partial U_1(r,0)/\partial r>0$ for $r\in(0,1)$. Therefore the potential $U_1(r,0)$ has only one critical point on the rays corresponding to even $k$, namely the origin.
\end{proof}
Lemma $\ref{crtianglesbeta1}$ implies that all the critical points lie on the rays stemming from the origin and bisecting the edges of the regular polygon, i.e. for 
$\theta=\pi k/n, k\in\mathbb Z$, with odd $k$.

Let us consider a restriction of the potential $U_1(r,\theta)$ to the perpendicular bisectors of the sides. Without loss of generality
we can assume $k=1$ and set $v(r):=U_1(r,\pi/n)$. Then
\begin{equation}\label{potenv}
v(r)=\frac{n}{1-r^2}\frac{1-r^n}{1+r^n}.
\end{equation}
To find the critical points of $v(r)$, we differentiate $(\ref{potenv})$:
\begin{equation}\label{derv}
v'(r)=2n\frac{r(1-r^{2n}-nr^{n-2}+nr^n)}{(1-r^2)^2(1+r^n)^2}.
\end{equation}
We clearly see that there is a critical point of the potential $v(r)$ located at the origin. Now consider the polynomial
$p_n(r)=-r^{2n}+nr^n-nr^{n-2}+1$. Then $p_n'(r)=-2nr^{2n-1}+n^2r^{n-1}-n(n-2)r^{n-3}$, and $p_n''(r)=-2n(2n-1)r^{2n-2}+
n^2(n-1)r^{n-2}-n(n-2)(n-3)r^{n-4}$. Note that $p_n(1)=p_n'(1)=0$ and $p_n''(1)=-4n\neq 0$ for all $n\geq3$. This shows that $r=1$ is a zero
of multiplicity 2 of the polynomial $p_n(r)$. Now observe that by Decartes's rule polynomial $p_n(r)$ has exactly three
positive roots. Next, we will need the following:
\begin{lemma}\label{interlemma}
For all $r>1$ and for all $n\geq3$ $p_n(r)<0$. 
\end{lemma}
\begin{proof}
Let $r>1$ and $n\geq3$. First we show that $p_n'(r)<0$. As $p_n'(r)=-2nr^{2n-1}+n^2r^{n-1}-n(n-2)r^{n-3}$, we need to show that $2nr^{2n-1}-n^2r^{n-1}+n(n-2)r^{n-3}>0$. 
That is the same as $n-2>nr^2-2r^{n+2}$. Consider the function $\psi(r)=nr^2-2r^{n+2}$. Then $\psi'(r)=2nr-2(n+2)r^{n+1}=2r(n-(n+2)r^n)<0$ for $r>1$, so that the function $\psi(r)$ is strictly decreasing. Hence $\psi(1)>\psi(r)$ for $r>1$. But $\psi(1)=n-2$, so $n-2>nr^2-2r^{n+2}$, as claimed. This tells us that $p_n'(r)<0$ for $r>1$, as desired.
Thus we see that $p_n(r)$ is a strictly decreasing function of $r$ for $r>1$. Therefore $p_n(1)>p_n(r)$ for $r>1$, which implies that $p_n(r)<0$. This concludes the proof.
\end{proof}
Since by the Lemma $\ref{interlemma}$ $p_n(r)<0$ for all $r>1$ and all $n\geq3$, we conclude that $p_n(r)$ has exactly one simple zero on $(0,1)$.
Thus the potential $v(r)$ has exactly one critical point on $(0,1)$. This completes the proof of the Theorem $\ref{critptsbeta1}$.
\qed

At this point we note that in the case $\beta=1$ it is quite easy to show that the potential is in fact a Morse function. Indeed, we have the following 
\begin{lemma}\label{morse}
The potential $v(r)$ is a Morse function, that is its critical points are non-degenerate.
\end{lemma}
\begin{proof} 
We compute the first and the second derivatives as follows:
\begin{empheq}{align*}
v'(r)&=2nrp_n(r)(1-r^2)^{-2}(1+r^n)^{-2},\\
v''(r)&=2n(1-r^2)^{-4}(1+r^n)^{-4}\\
&\times\{(p_n(r)+rp_n'(r))(1-r^2)^2(1+r^n)^2+rp_n(r)((1-r^2)^2(1+r^n)^2)'\}.
\end{empheq}
Suppose that $r_0\neq0$ is a critical point of $v(r)$, that is $v'(r_0)=0$. Then $p_n(r_0)=0$ for all $n\geq3$, and
\begin{empheq}{align*}
& v''(r_0)=2n(1-r_0^2)^{-4}(1+r_0^n)^{-4}\times\\
& \{(p_n(r_0)+r_0p_n'(r_0))(1-r_0^2)^2(1+r_0^n)^2+r_0p_n(r_0)((1-r^2)^2(1+r^n)^2)'(r_0)\}.
\end{empheq}
so that
\begin{equation*}
v''(r_0)=\frac{2nr_0p_n'(r_0)}{(1-r_0^2)^2(1+r_0^n)^2}\neq0.
\end{equation*}
It is clear from the expressions for $v'(r)$ and $v''(r)$ that $r=0$ is a critical point for $v(r)$ and that it is not degenerate. Therefore all critical points of $v(r)$ are non-degenerate, which means that $v(r)$ is Morse on $[0,1)$.
\end{proof}

{\bf Proof of Theorem $\ref{extres}$.} First we will need a few facts from the theory of real analytic functions \cite{KrantzParks}. In particular, we will need
\begin{proposition}[Identity Theorem]
Let $D\subset\mathbb R^m$ be connected and let $f$ be a real-analytic function on $D$. If there is a non-empty set $U\subset D$ such that $f(x)=0$ for all $x\in U$, then $f\equiv 0$ on $D$.
\end{proposition}

We will also be making use of 

\begin{proposition}[Real Analytic Implicit Function Theorem]
Suppose $F:\mathbb R^2\rightarrow \mathbb R$ is real analytic in a neghborhood of $(x_0,y_0)$ for some $x_0\in\mathbb R$ and some $y_0\in\mathbb R$. If $F(x_0,y_0)=0$ and
\begin{equation*}
\frac{\partial F(x_0,y_0)}{\partial y}\neq0,
\end{equation*}
then there exists a function $h:\mathbb R\rightarrow \mathbb R$ which is real analytic in a neighborhood of $x_0$ and such that
\begin{equation*}
F(x,h(x))=0
\end{equation*}
holds in a neighborhood of $x_0$.
\end{proposition}

Let us now briefly mention a few standard facts from the theory of stable mappings that will be used in the content of the discussion below.

Let $X$ and $Y$ be smooth manifolds. Denote by $C^{\infty}(X,Y)$ the set of smooth mappings from $X$ to $Y$. The set $C^{\infty}(X,Y)$ equipped with Whitney topology becomes a topological space (for details see \cite{GolubG}).\newline\indent
Let $\varphi$ and $\psi$ be elements of $C^{\infty}(X,Y)$. We will call $\varphi$ and $\psi$ {\it isotopic} if there exist diffeomophisms $f:X\rightarrow X$ and $g:Y\rightarrow Y$, each homotopic to the identity on their respective space, such that $\varphi=g\circ\psi\circ f$. An element $\varphi$ of $C^{\infty}(X,Y)$ will be called {\it stable} if there is a neighborhood $W_{\varphi}$ of $\varphi$ in $C^{\infty}(X,Y)$ such that each $\psi$ in $W_{\varphi}$ is isotopic to $\varphi$. Note that the stable mappings in $C^{\infty}(X,Y)$ always form an open subset.\newline\indent
Now employ the following observation, which can be found, for example, in \cite{Johnson}. Assume that two stable mappings in $C^{\infty}(X,Y)$ are connected by a path $\gamma$ consisting of stable mappings. Then we can cover $\gamma$ by a finite collection of open sets such that any two mappings in each set are isotopic. Using an induction argument it follows that any two mappings in $\gamma$ are isotopic. Recalling that any two isotopic mappings have the same number of critical points, it follows that any two mappings in $\gamma$ will have the same number of critical points.\newline\indent
Suppose $X$ is a compact manifold and let $\varphi$ be an element of  $C^{\infty}(X,\mathbb R)$. Then $\varphi$ is stable if and only if $\varphi$ is a Morse function whose critical values are pairwise distinct. In particular, it follows that if we have two Morse functions with distinct critical values connected by a path $\gamma$ in $C^{\infty}(X,\mathbb R)$ consisting of Morse functions with distinct critical values, any two functions in $\gamma$ will have the same number of the critical points.

We also note \cite{GuiPol}[Ex. 19, p. 47] that if $f$ is a Morse function with not necessarily distinct critical values, we can find a function $\tilde{f}$ that has the same critical points as $f$ and is arbitrarily close to $f$ in the $C^2$ topology. Moreover, the critical values of $\tilde{f}$ are distinct.

Now prove the following easy but important fact:

\begin{lemma}\label{analytofpot}
The potential $u_\beta(r)$ is a real analytic function in $\beta, r$.
\end{lemma}
\begin{proof}
Note that $(1+r^2-2r\cos(\theta_j))>0$ is an analytic function of $r$. Hence $\log(1+r^2-2r\cos(\theta_j))$ is also analytic in $r$. Next, observe that
$(1+r^2-2r\cos(\theta_j))^{-\beta}=\exp(-\beta\log(1+r^2$ $-2r\cos(\theta_j))$, which is an analytic function in $\beta, r$.
\end{proof}

We are now ready to state our extension result for the potential $u_\beta(r)$.
\begin{lemma}[Extension Theorem for $u_\beta(r)$]
There exist real numbers $\beta_l$ and $\beta_r$ such that $1\in(\beta_l,\beta_r)$ so that for each $\beta\in(\beta_l,\beta_r)$ the potential $u_\beta(r)$ has a unique critical point different from the origin, as well as a critical point at the origin.
\end{lemma}
\begin{proof}
Let $g:\mathbb R^2\rightarrow \mathbb R$ be defined as $g(\beta,r):=u'_\beta(r)$. By Lemma $\ref{analytofpot}$ the potential $u_\beta(r)$ is analytic in $\beta$, $r$, so its derivatives with respect to $\beta$ and $r$ are also analytic in $\beta$, $r$. This implies that the function $g(\beta,r)$ is analytic in $\beta$, $r$. Let $r_0$ be a non-zero critical point of the potential $u_{\beta_0}(r)$ for $\beta_0=1$. By Lemma $\ref{morse}$ we know that $r_0$ is a non-degenerate critical point, that is $u''_{\beta_0}(r_0)\neq0$. Therefore
\begin{equation*}
g(\beta_0,r_0)=0, \quad \frac{\partial g(\beta_0,r_0)}{\partial r}\neq0.
\end{equation*}
By the Real Analytic Implicit Function Theorem there exist a neighborhood $U_0=(a,b)\times(c,d)$ of $(\beta_0,r_0)$ and an analytic function $h:\mathbb R\rightarrow \mathbb R$ such that $\{(\beta,r)\in U_0: g(\beta,r)=0\}=\{(\beta,h(\beta)):\beta\in (a,b)\}$. That is, $u'_\beta(h(\beta))=0$ in $(a,b)$, and $h(\beta)$ is the only such a root of $u'_\beta(r)=0$.

We conclude that for each $\beta$ in a neighborhood of $\beta_0=1$ there exists a unique critical point of $u_\beta(r)$, call it $r_\beta$, and $r_\beta$ depends analytically on $\beta$.

Now observe that from the Identity Theorem it follows that an analytic function on a closed interval can only have a finite number of zeros. Therefore $\partial g(\beta,r)/\partial r$ has a finite number of zeros. One may shrink $U_0$ such that $\partial g(\beta,r)/\partial r\neq0$ on $U_0$ (since $\partial g(\beta,r)/\partial r\neq0$ at $(\beta_0,r_0)$). Denote an interval of such $\beta$'s as $(\beta_l,\beta_r)$. Then all critical points $r_\beta$ are non-degenerate for $\beta\in(\beta_l,\beta_r)$. Note that $1\in(\beta_l,\beta_r)$.

After these preliminary remarks let us make the following observation. Since we know that the potential $u_\beta(r)$ has no critical points on $(\cos(\pi/n), 1)$, we can assume that $r\in[0,\cos(\pi/n)]$ and $\beta\in(\beta_l,\beta_r)$. It is clear that $u_\beta\in C^{\infty}([0,\cos(\pi/n)],\mathbb R)$ for all $\beta\in(\beta_l,\beta_r)$.

Next, note that  for each $\beta\in(\beta_l,\beta_r)$ the potential $u_\beta(r)$ is a Morse function. Hence $u_\beta$ defines a path in $C^{\infty}([0,\cos(\pi/n)],\mathbb R)$. Therefore by above considerations any two functions in the path defined by $u_\beta$ have the same number of critical points. As we have shown that $u_\beta(r)$ for $\beta=1$ has exactly one critical point (excluding the origin), and $1\in(\beta_l,\beta_r)$, it follows that for any $\beta\in(\beta_l,\beta_r)$ the potential $u_\beta(r)$ also has exactly one critical point (excluding origin). The proof of the lemma is complete.
\end{proof}
Recalling that the potential $u_\beta(r)$ is a restriction of the 2D potential $U_\beta(r,\theta)$ to the perpendicular bisectors to the sides, the statement of the theorem follows.\qed

{\bf Proof of Example {\ref{ex3}}}

In this case the Riesz potential on the bisectors is of the form
\begin{equation*}
u_\beta(r)=\frac{1}{(1+r)^{2\beta}}+\frac{2}{(1+r^2-r)^\beta},\quad r\in(0,1).
\end{equation*}
Its derivative is
\begin{empheq}{align*}
u'_\beta(r) & = \frac{-2\beta}{(1+r)^{2\beta+1}}+\frac{-2\beta(2r-1)}{(1+r^2-r)^{\beta+1}}\\
& = \frac{-2\beta}{(1+r)^{2\beta+1}(1+r^2-r)^{\beta+1}}((2r-1)(1+r)^{2\beta+1}+(1+r^2-r)^{\beta+1})\\
& = -2\beta(1+r)^{-(2\beta+1)}(1+r^2-r)^{-(\beta+1)}f_\beta(r),
\end{empheq}
where $f_\beta(r):=(1+r)^{2\beta+1}(2r-1)+(1+r^2-r)^{\beta+1}$. Differentiating, we obtain
\begin{equation*}
f'_\beta(r)=(2\beta+1)(1+r)^{2\beta}(2r-1)+2(1+r)^{2\beta+1}+(\beta+1)(1+r^2-r)^{\beta}(2r-1).
\end{equation*}
For the second derivative we have
\begin{empheq}{align*}
f''_\beta(r) & = (2\beta+1)(2\beta(1+r)^{2\beta-1}(2r-1)+2(1+r)^{2\beta})\\
& + 2(2\beta+1)(1+r)^{2\beta}\\
& + (\beta+1)(\beta(1+r^2-r)^{\beta-1}(2r-1)^2+2(1+r^2-r)^\beta)\\
& = 2(2\beta+1)(1+r)^{2\beta-1}(\beta(2r-1)+(1+r))\\
& + 2(2\beta+1)(1+r)^{2\beta}\\
& + (\beta+1)(\beta(1+r^2-r)^{\beta-1}(2r-1)^2+2(1+r^2-r)^\beta)\\
& = 2(2\beta+1)(1+r)^{2\beta-1}((2\beta+1)r+(1-\beta))\\
& + 2(2\beta+1)(1+r)^{2\beta}\\
& + (\beta+1)(\beta(1+r^2-r)^{\beta-1}(2r-1)^2+2(1+r^2-r)^\beta).\\ 
\end{empheq}
As $\beta\in(0,1)$, it follows that $(1-\beta)>0$, and we see from the above that $f''_\beta>0$ for all $r\in(0,1)$ and all $\beta\in(0,1)$. That says that $f_\beta$ is convex on $(0,1)$ for all $\beta\in(0,1)$.

Also observe that $f_\beta(0)=-1+1=0$. It is not hard to see that $f_\beta(1/2)=(3/4)^{\beta+1}>0$, and a simple calculation shows $f_\beta(\beta/(\beta+1))<0$ for all $\beta\in(0,1)$.

The above facts combined show that $f_\beta$ has a unique positive root on $[0,1/2]$.

But the roots of $f_\beta(r)$ are exactly the critical points of the potential $u_\beta(r)$, and for the case $n=3$ we know that the non-trivial critical points are located on $(0,\cos(\pi/3)]=(0,1/2]$. Hence we have shown that the potential $u_\beta$ has a unique non-trivial critical point for all $\beta\in(0,1)$. Recalling that the potential $u_\beta(r)$ is a restriction of the 2D Riesz potential $U_\beta(r,\theta)$ to the perpendicular bisectors, we conclude that $U_\beta(r,\theta)$ has exactly four critical points in the case of an equilateral triangle for {\it all} $\beta\in(0,1)$, as claimed.
\qed

\section*{\large \textbf{Acknowledgements}}
The author would like to thank his doctoral advisor Prof. Igor E. Pritsker for suggesting the problem and his constant attention to the paper. The author also wishes to thank Prof. Jesse Johnson and Prof. Roger Zierau for valuable discussions.



\end{document}